\documentclass{amsart}
\usepackage{setspace}
\usepackage{a4}
\usepackage{amssymb,amsmath,amsthm,latexsym}
\usepackage{amsfonts}
\usepackage{amsfonts}
\usepackage{graphicx}
\usepackage{textcomp}
\usepackage{cite}
\usepackage{enumerate}
\usepackage[mathscr]{euscript}
\usepackage{mathtools}
\newtheorem{theorem}{Theorem}[section]

\newtheorem{corollary}[theorem] {Corollary}
\newtheorem{definition}[theorem]{Definition}
\newtheorem{example}[theorem]{Example}

\newtheorem{problem}[theorem]{Problem}

\newtheorem{remark}[theorem]{Remark}
\title{This is the title}
\usepackage{amssymb}
\usepackage{amssymb}
\usepackage{amsmath}
\usepackage{tikz}
\usepackage{hyperref}
\usepackage{enumerate}
\usepackage{mathtools}
\usepackage{amsmath}
\usepackage{tikz}
\usepackage{amssymb}
\usepackage{amsmath}
\usepackage{tikz}
\usepackage{hyperref}
\usepackage{amssymb}
\usepackage{enumerate}
\usepackage{mathtools}
\usepackage{amsmath}
\usepackage{tikz}
\usepackage{graphicx}
\usepackage{textcomp}
\usetikzlibrary{cd}
\usepackage{fancyhdr}

\begin{document}
	\hrule\hrule\hrule\hrule\hrule
	\vspace{0.3cm}	
	\begin{center}
		{\bf{p-ADIC MAGIC CONTRACTIONS, p-ADIC VON NEUMANN INEQUALITY  AND p-ADIC SZ.-NAGY DILATION}}\\
		\vspace{0.3cm}
		\hrule\hrule\hrule\hrule\hrule
		\vspace{0.3cm}
		\textbf{K. MAHESH KRISHNA}\\
		Post Doctoral Fellow \\
		Statistics and Mathematics Unit\\
		Indian Statistical Institute, Bangalore Centre\\
		Karnataka 560 059, India\\
		Email: kmaheshak@gmail.com\\
		
		Date: \today
	\end{center}

\hrule
\vspace{0.5cm}
\textbf{Abstract}: We introduce the notion of p-adic magic contraction on p-adic Hilbert space. We derive p-adic Halmos dilation, p-adic Egervary N-dilation, p-adic von Neumann inequality and p-adic Sz.-Nagy dilation for p-adic magic contraction.


\vspace{0.5cm}
\hrule 

\section{Introduction}
Dilation theory of contractions on Hilbert spaces which started from the work of Halmos \cite{HALMOS}  and forwarded by Sz.-Nagy \cite{NAGY}  is now 70 years old. Influential works on this area are documented in references \cite{SCHAFFER, NAGYLIFTING, ANDO, PAULSENBOOK, PISIERBOOK, SARASON, AMBROZIEAMULLER, AGLERMCCARTHY, DOUGLAS, NAGYFOIAS, LEVYSHALIT, ARVESON, ORRGUIDED, FRAZHO, FOIASFRAZHO, FOIASFRAZHOGOHBERGKAASHOEK, DURSZTNAGY, POPESCU, BERCOVICI, BHATMUKHERJEE, EGERVARY, BHATTACHARYYA, NAGY1960, PARROTT, DRURY, CRABBDAVIE, MCCARTHYSHALIT, VAROPOULOS, CHOIDAVIDSON} Gradually,  the theory has been put in the Banach space settings \cite{FACKLERGLUCK, AKCOGLUSUCHESTON, STROESCU, NAGELPALM, AKCOGLUKOPP, KERNNAGELPALM, FENDLER}.

Very recently, the dilation theory has been introduced for functions on sets \cite{BHATDERAKSHITH} and linear operators on vector spaces  \cite{HANLARSONLIULIU, BHATDERAKSHITH, KRISHNAJOHNSON}. In this paper,  we introduce the notion of magic contraction (Definition \ref{PS}). We then derive p-adic versions of Halmos dilation (Theorem \ref{HD}), Egervary N-dilation (Theorem \ref{ED}), von Neumann inequality (Theorem \ref{VE}), Sz.-Nagy dilation (Theorem \ref{ND}) and von Neumann ergodic result (Theorem \ref{PADICERGODIC}). Our paper is highly motivated from the paper of Halmos \cite{HALMOS}, Egervary \cite{EGERVARY}, Schaffer \cite{SCHAFFER}, Sz.-Nagy \cite{NAGY}, Bhat, De and Rakshit \cite{BHATDERAKSHITH} and Krishna and Johnson \cite{KRISHNAJOHNSON}.

\section{p-adic Magic Contractions, p-adic von Neumann Inequality and p-adic Sz.-Nagy Dilation}
We use the following notion of  p-adic Hilbert space which is slight variant of notion  introduced by Kalisch \cite{KALISCH}. 
\begin{definition}\cite{KALISCH}\label{PADICDEF}
Let $\mathbb{K}$ be a non-Archimedean (complete) valued field (with valuation $|\cdot|$) and $\mathcal{X}$ be a non-Archimedean Banach space (with norm $\|\cdot\|$) over $\mathbb{K}$. We say that $\mathcal{X}$ is a p-adic Hilbert space if there is a map (called as inner product) $\langle \cdot, \cdot \rangle: \mathcal{X} \times \mathcal{X} \to \mathbb{K}$ satisfying following.
\begin{enumerate}[\upshape (i)]
	\item If $x \in \mathcal{X}$ is such that $\langle x,y \rangle =0$ for all $y \in \mathcal{X}$, then $x=0$.
	\item $\langle x, y \rangle =\langle y, x \rangle$ for all $x,y \in \mathcal{X}$.
	\item $\langle \alpha x+y, z \rangle =\alpha \langle x,  z \rangle +\langle y,z \rangle $ for all  $\alpha  \in \mathbb{K}$, for all $x, y, z \in \mathcal{X}$.
	\item $|\langle x, y \rangle |\leq \|x\|\|y\|$ for all $x,y \in \mathcal{X}$.
\end{enumerate}
\end{definition}
Following are  standard examples we keep in mind.
\begin{example}
Let $d\in \mathbb{N}$ and 	$\mathbb{K}$ be a non-Archimedean (complete) valued field. Define
\begin{align*}
	\mathbb{K}^d\coloneqq \{(x_j)_{j=1}^d:x_j \in \mathbb{K}, 1\leq j \leq d\}
\end{align*} 
Then $\mathbb{K}^d$ is a p-adic Hilbert space w.r.t. norm 
\begin{align*}
	\|(x_j)_{j=1}^d\|\coloneqq \max_{1\leq j \leq d}|x_j|, \quad \forall (x_j)_{j=1}^d\in 	\mathbb{K}^d
\end{align*}
and inner product 
\begin{align*}
	\langle (x_j)_{j=1}^d, (y_j)_{j=1}^d\rangle \coloneqq \sum_{j=1}^dx_jy_j, \quad \forall (x_j)_{j=1}^d, (y_j)_{j=1}^d\in 	\mathbb{K}^d.
\end{align*}
\end{example}
\begin{example}
Let $\mathbb{K}$ be a non-Archimedean (complete) valued field. Define
\begin{align*}
	c_0(\mathbb{N}, \mathbb{K})\coloneqq \{(x_n)_{n=1}^\infty:x_n \in \mathbb{K}, \forall n \in \mathbb{N}, \lim_{n\to \infty}|x_n|=0\}
\end{align*} 
Then $c_0(\mathbb{N}, \mathbb{K})$ is a p-adic Hilbert space w.r.t. norm 
\begin{align*}
	\|(x_n)_{n=1}^\infty\|\coloneqq \sup_{n\in \mathbb{N}}|x_n|, \quad \forall (x_n)_{n=1}^\infty\in 	c_0(\mathbb{N}, \mathbb{K})
\end{align*}
and inner product 
\begin{align*}
	\langle (x_n)_{n=1}^\infty, (y_n)_{n=1}^\infty\rangle \coloneqq \sum_{n=1}^\infty x_ny_n, \quad \forall (x_n)_{n=1}^\infty, (y_n)_{n=1}^\infty\in 	c_0(\mathbb{N}, \mathbb{K}).
\end{align*}	
\end{example}
Let $\mathcal{X}$ be a p-adic Hilbert space and $T:\mathcal{X}\to \mathcal{X}$ be a bounded linear operator. We say that $T$ is adjointable if there is a bounded linear operator, denoted by $T^*:\mathcal{X}\to \mathcal{X}$ such that $\langle Tx,y\rangle =\langle x,T^*y\rangle$, $\forall x, y \in \mathcal{X}$. Note that (i) in Definition \ref{PADICDEF} says that adjoint, if exists,  is unique. An adjointable bounded linear operator $U$ is said to be a unitary if $UU^*=U^*U=I_\mathcal{X}$, the identity operator on $\mathcal{X}$.  An adjointable bounded linear operator $P$ is said to be projection if $P^2=P^*=P$. An adjointable bounded linear operator $T$ is said to be an isometry if $T^*T=I_\mathcal{X}$. An adjointable bounded linear operator $T$ is said to be  self-adjoint  if $T^*=T$. We denote the identity operator on $\mathcal{X}$ by $I_\mathcal{X}$. Following is the magic definition.
\begin{definition}\label{PS}
Let $\mathcal{X}$ be a p-adic Hilbert space and $T:\mathcal{X}\to \mathcal{X}$ be a bounded linear adjointable operator. We say that 	$T$ is a \textbf{magic contraction} if there are self adjoint bounded linear operators $M_{T}:\mathcal{X}\to \mathcal{X}$ and $ M_{T^*}: \mathcal{X}\to \mathcal{X}$ (which may not be unique, that is why $M$) such that 
\begin{align*}
	&M_T^2=I_\mathcal{X}-T^*T, \quad M_{T^*}^2=I_\mathcal{X}-TT^*,\\
	&TM_T= M_{T^*}T.
\end{align*}
\end{definition}
Our first result is the p-adic Halmos dilation. 
\begin{theorem} (\textbf{p-adic Halmos dilation})\label{HD}
    	Let $\mathcal{X}$ be a p-adic Hilbert  space  and 	$T: \mathcal{X} \to \mathcal{X}$ be a magic contraction. Then the operator 
    	\begin{align*}
    		U\coloneqq \begin{pmatrix}
    			T & M_{T^*}   \\
    			M_T & -T^*  \\
    		\end{pmatrix}
    	\end{align*}
    	is unitary  on 	$\mathcal{X}\oplus \mathcal{X}$. In other words,
    	\begin{align*}
    		T=P_\mathcal{X}U|_\mathcal{X}, \quad 	T^*=P_\mathcal{X}U^*|_\mathcal{X},
    	\end{align*}
    	where $P_\mathcal{X}:\mathcal{X}\oplus \mathcal{X}\ni (x, y) \mapsto x \in \mathcal{X}$.
    \end{theorem}
\begin{proof}
A direct calculation says that 	
\begin{align*}
	V\coloneqq \begin{pmatrix}
		T^* & M_T   \\
		M_{T^*} & -T  \\
	\end{pmatrix}
\end{align*}
is the inverse and adjoint of $U$.	
\end{proof}
As an application of dilation, Sz.-Nagy gave an easy proof of fixed point of a contraction is also a fixed point of its adjoint \cite{NAGY}. Here we give a similar result for p-adic magic contractions. 
\begin{corollary}
	Let $\mathcal{X}$ be a p-adic Hilbert  space  and 	$T: \mathcal{X} \to \mathcal{X}$ be a magic contraction. If 	$x \in \mathcal{X}$ is such that $Tx=x$, then $T^*x=x$.
\end{corollary}
\begin{proof}
	Let $U$ be a Halmos dilation of $T$. Then $	x=Tx=P_\mathcal{X}Ux$. Since $P_\mathcal{X}$ is an orthogonal projection, we must have $Ux=x$. Since $U$ is unitary, we then have $U^*x=x$. Therefore $T^*x=P_\mathcal{X} U^*x=P_\mathcal{X}x=x$.
\end{proof}
Our second result is the p-adic Egervary dilation.
\begin{theorem} \label{ED}(\textbf{p-adic Egervary N-dilation})
Let $\mathcal{X}$ be a p-adic Hilbert space and $T:\mathcal{X}\to \mathcal{X}$ be a magic contraction. Let $N$ be a natural number. Then the operator 
\begin{align*}
U\coloneqq \begin{pmatrix}
T & 0& 0 & \cdots &0&0 & M_{T^*}   \\
M_T & 0& 0 & \cdots &0&0& -T^*   \\
0&I_\mathcal{X}&0&\cdots &0&0& 0\\
0&0&I_\mathcal{X}&\cdots &0&0 & 0\\
\vdots &\vdots &\vdots & &\vdots & \vdots &\vdots \\
0&0&0&\cdots &0&0 & 0\\
0&0&0&\cdots &I_\mathcal{X}&0 & 0\\
0&0&0&\cdots &0&I_\mathcal{X} & 0\\
\end{pmatrix}_{(N+1)\times (N+1)}
\end{align*}
is unitary on 	$\oplus_{k=1}^{N+1} \mathcal{X}$ and 
\begin{align}\label{FINITEDILATIONEQUATION}
T^k=P_\mathcal{X}U^k|_\mathcal{X},\quad \forall k=1, \dots, N, \quad (T^*)^k=P_\mathcal{X}(U^*)^k|_\mathcal{X},\quad \forall k=1, \dots, N,
\end{align}
where $P_\mathcal{X}:\oplus_{k=1}^{N+1} \mathcal{X} \ni (x_k)_{k=1}^{N+1} \mapsto x_1 \in \mathcal{X}$. 		
\end{theorem}
\begin{proof}
A direct calculation of power of $U$ gives Equation (\ref{FINITEDILATIONEQUATION}). To complete the proof, now we need show that $U$ is unitary. Define 
\begin{align*}
	V\coloneqq \begin{pmatrix}
		T^* & M_T& 0 & \cdots &0&0 & 0   \\
		0 & 0& I_\mathcal{X} & \cdots &0&0& 0   \\
		0&0&0&\cdots &0&0& 0\\
		0&0&0&\cdots &0&0 & 0\\
		\vdots &\vdots &\vdots & &\vdots & \vdots &\vdots \\
		0&0&0&\cdots &0&I_\mathcal{X} & 0\\
		0&0&0&\cdots &0&0 & I_\mathcal{X}\\
		M_{T^*}&-T&0&\cdots &0&0 & 0\\
	\end{pmatrix}_{(N+1)\times (N+1)}.
\end{align*}
Then $UV=VU=I_{\oplus_{k=1}^{N+1} \mathcal{X}}$ and  $U^*=V$.
\end{proof}
  Note that the Equation (\ref{FINITEDILATIONEQUATION}) holds only upto $N$ and not for $N+1$ and higher natural numbers. Inspired from the arguments of Sz.-Nagy \cite{NAGY} for the proof of classical von Neumann inequality, we derive following p-adic von Neumann inequality.
  \begin{theorem}\label{VE}
  	(\textbf{p-adic von Neumann inequality}) Let $f(z)\coloneqq a_0+a_1z+\cdots +a_Nz^N \in \mathbb{K}[z]$ be a polynomial of degree $N$ over $\mathbb{K}$. Let $T:\mathcal{X}\to \mathcal{X}$ be a magic contraction. Let $U:\oplus_{k=1}^{N+1} \mathcal{X} \to \oplus_{k=1}^{N+1} \mathcal{X}$ be any Egervary N-dilation of $T$. Then 
  	\begin{align*}
  		\|f(T)\|\leq \|f(U)\|.
  	\end{align*}
  \end{theorem}
\begin{proof}
We have 
\begin{align*}
f(T)&=a_0I_\mathcal{X}+a_1T+\cdots +a_NT^N=a_0P_\mathcal{X}+a_1P_\mathcal{X}U|_\mathcal{X}+\cdots +a_NP_\mathcal{X}U^N|_\mathcal{X}\\
&=P_\mathcal{X}(a_0I_\mathcal{X}+a_1U|_\mathcal{X}+\cdots +a_NU^N|_\mathcal{X})=P_\mathcal{X}f(U)|_\mathcal{X}	
\end{align*}
Therefore 
\begin{align*}
	\|f(T)\|=\|P_\mathcal{X}f(U)|_\mathcal{X}\|\leq \|P_\mathcal{X}\|\|f(U)|_\mathcal{X}\|=\|f(U)|_\mathcal{X}\|\leq \|f(U)\|.
\end{align*}	
\end{proof}
 In the following theorem, given a p-adic Hilbert space $\mathcal{X}$,  $\oplus_{n=-\infty}^{\infty} \mathcal{X}$ is the p-adic Hilbert  space defined by 
\begin{align*}
\oplus_{n=-\infty}^{\infty} \mathcal{X}\coloneqq \{ \{x_n\}_{n=-\infty}^\infty, x_n \in \mathcal{X}, \forall n \in \mathbb{Z}, \lim_{|n|\to \infty}\|x_n\|=0\}
\end{align*}
equipped with norm 
\begin{align*}
	\| \{x_n\}_{n=-\infty}^\infty\|\coloneqq \sup_{n\in \mathbb{Z}}\|x_n\|, \quad \forall \{x_n\}_{n=-\infty}^\infty \in \oplus_{n=-\infty}^{\infty} \mathcal{X}
\end{align*}
and inner product 
\begin{align*}
	\langle \{x_n\}_{n=-\infty}^\infty, \{y_n\}_{n=-\infty}^\infty\rangle \coloneqq \sum_{n=-\infty}^{\infty}\langle x_n, y_n \rangle , \quad \forall \{x_n\}_{n=-\infty}^\infty, \{y_n\}_{n=-\infty}^\infty\in \oplus_{n=-\infty}^{\infty} \mathcal{X}.
\end{align*}
Following is the most important p-adic Sz.-Nagy dilation.
\begin{theorem}\label{ND} (\textbf{p-adic Sz.-Nagy dilation})
Let $\mathcal{X}$ be a p-adic Hilbert  space  and 	$T: \mathcal{X} \to \mathcal{X}$ be a magic contraction.  Let $U\coloneqq(u_{n,m})_{-\infty \leq n,m\leq \infty}$ be the operator defined on 
$\oplus_{n=-\infty}^{\infty} \mathcal{X}$ given by  the infinite matrix defined as follows:
\begin{align*}
&u_{0,0}\coloneqq T, \quad u_{0,1}\coloneqq M_{T^*}, \quad u_{-1, 0}\coloneqq M_T, \quad u_{-1, 1}\coloneqq -T^*, \\
& u_{n,n+1}\coloneqq I_\mathcal{X}, \quad \forall n \in \mathbb{Z}, n\neq 0,1,  \quad u_{n,m}\coloneqq 0 \quad  \text{ otherwise},
\end{align*}
i.e., 
\begin{align*}
U=\begin{pmatrix}
 &\vdots &\vdots & \vdots & \vdots & \vdots &\vdots & \\
 \cdots & I_\mathcal{X}&0&0&0& 0&0&\cdots &\\
\cdots & 0 & I_\mathcal{X} & 0 & 0&  0&0& \cdots & \\
\cdots & 0 & 0& M_T & -T^*& 0&0&\cdots  & \\
\cdots & 0&0&\boxed{T}&M_{T^*}& 0&0&\cdots&\\
\cdots & 0&0&0&0& I_\mathcal{X}& 0&\cdots &\\
\cdots & 0&0&0&0& 0&I_\mathcal{X}&\cdots &\\
 & \vdots &\vdots &\vdots &\vdots  &\vdots & \vdots & \\
\end{pmatrix}_{\infty\times \infty}
\end{align*}
where $T$ is in the $(0,0)$  position (which is boxed), is unitary   on 	$\oplus_{n=-\infty}^{\infty} \mathcal{X}$ and 
\begin{align}\label{INFINITEDILATIONEQUATION}
T^n=P_\mathcal{X}U^n|_\mathcal{X},\quad \forall n\in \mathbb{N}, \quad (T^*)^n=P_\mathcal{X}(U^*)^n|_\mathcal{X},\quad \forall n\in \mathbb{N},
\end{align}
where $P_\mathcal{X}:\oplus_{n=-\infty}^{\infty} \mathcal{X}\ni  (x_n)_{n=-\infty}^{\infty} \mapsto x_0 \in \mathcal{X}$.
\end{theorem}
\begin{proof}
We  get Equation (\ref{INFINITEDILATIONEQUATION}) by calculation of powers of $U$. The matrix   $V\coloneqq(v_{n,m})_{-\infty \leq n,m\leq \infty}$ defined by  
\begin{align*}
	&v_{0,0}\coloneqq T^*, \quad v_{0,-1}\coloneqq M_{T}, \quad v_{1, 0}\coloneqq M_{T^*}, \quad v_{1, -1}\coloneqq T, \\
	& v_{n,n-1}\coloneqq I_\mathcal{X}, \quad \forall n \in \mathbb{Z}, n\neq 0,1,  \quad v_{n,m}\coloneqq 0 \quad  \text{ otherwise},
\end{align*}
i.e., 
\begin{align*}
V=\begin{pmatrix}
&\vdots &\vdots &\vdots & \vdots & \vdots & \vdots & \\
\cdots&I_\mathcal{X} & 0 & 0& 0 & 0&  0& \cdots & \\
\cdots &0& I_\mathcal{X} & 0& 0 & 0&  0& \cdots & \\
\cdots &0& 0 & M_T& \boxed{T^*} & 0& 0&\cdots  & \\
\cdots &0& 0&-T&M_{T^*}&0& 0&\cdots&\\
\cdots &0& 0&0&0&I_\mathcal{X}& 0&\cdots &\\
\cdots &0& 0&0&0&0& I_\mathcal{X}&\cdots &\\
&\vdots &\vdots &\vdots &\vdots &\vdots  & \vdots & \\
\end{pmatrix}_{\infty\times \infty}
\end{align*}
where $T^*$ is in the $(0.0)$  position (which is boxed), satisfies $UV=VU=I_{\oplus_{n=-\infty}^{\infty} \mathcal{X}}$ and  $U^*=V$. 
\end{proof}
We note that explicit sequential form of $U$ is
\begin{align*}
	U(x_n)_{n=-\infty}^{\infty}=(\dots, x_{-2}, x_{-1}, M_Tx_0-T^*x_1, \boxed{Tx_0+M_{T^*}x_1}, x_2, x_2, \dots)
\end{align*}
where $T^*$ is in the $0$  position (which is boxed)
and $U^*$ is 
\begin{align*}
	U^*(x_n)_{n=-\infty}^{\infty}=(\dots, x_{-3}, x_{-2},  \boxed{M_Tx_{-1}+T^*x_0}, -Tx_{-1}+M_{T^*}x_0, x_1, \dots).
\end{align*}
Using his dilation result, Sz.-Nagy gave a new proof of von Neumann mean ergodic theorem \cite{NAGY}. Motivated from this,  now derive p-adic von Neumann mean ergodic theorem.
\begin{theorem} (\textbf{p-adic von Neumann mean ergodic theorem})\label{PADICERGODIC}
Let $\mathcal{X}$ be a p-adic Hilbert  space  and 	$T: \mathcal{X} \to \mathcal{X}$ be a magic contraction. If p-adic Sz.-Nagy dilation  $U$ of $T$ is such that the limit 
\begin{align*}
\lim_{N\to \infty } \frac{1}{N+1}\sum_{n=1}^{N}U^nx \quad \text{ exists for all } x \in \mathcal{X},	
\end{align*}
then 	the limit 
\begin{align*}
	\lim_{N\to \infty } \frac{1}{N+1}\sum_{n=1}^{N}T^nx \quad \text{ exists for all } x \in \mathcal{X}.	
\end{align*}
\end{theorem}
\begin{proof}
This follows from the observation 	
\begin{align*}
\lim_{N\to \infty } \frac{1}{N+1}\sum_{n=1}^{N}T^nx=\lim_{N\to \infty } \frac{1}{N+1}\sum_{n=1}^{N}P_\mathcal{X} U^nx=P_\mathcal{X}\left(\lim_{N\to \infty }\frac{1}{N+1}\sum_{n=1}^{N}U^nx\right) \quad \forall  x \in \mathcal{X}.		
\end{align*}
\end{proof}
We are now in the position to ask following problems based on dilation theory in Hilbert spaces.
\begin{problem}
\textbf{\begin{enumerate}[\upshape(i)]
		\item Whether there is p-adic Ando dilation? If yes, whether  one can dilate commuting three, four, ... commuting magic contractions to commuting unitaries?
		\item Whether there is p-adic von Neumann-Ando inequality?
		\item Whether there is (a kind of) uniqueness of p-adic Halmos dilation?
		\item Whether there is p-adic intertwining-lifting theorem (commutant lifting theorem)?
	\end{enumerate}}
\end{problem}
\begin{remark}
	Even though we derived all results in the p-adic setting, we can do all  the results except von Neumann inequality and von Neumann ergodic theorem, for modules (or even vector spaces) which admits bilinear (resp. conjugate) forms over  rings (resp. *-rings). Meanwhile, in that case, the title  of the paper can be written as MAGIC  CONTRACTIONS ON MODULES/VECTOR SPACES AND SZ.-NAGY DILATION. 
\end{remark}
	We give various  examples.
	\begin{example}
		  Let $\mathbb{Z}_3$ be the standard modulo $3$ field. Then the operator 
	\begin{align*}
	T\coloneqq 	 \begin{pmatrix}
		2 & 2   \\
		2 & 2  \\
	\end{pmatrix}: \mathbb{Z}_3^2 \ni (x,y)\mapsto T \begin{pmatrix}
	x\\
	y\\
\end{pmatrix} \in  \mathbb{Z}_3^2
	\end{align*} 
is a magic contraction. For, first notice
\begin{align*}
\begin{pmatrix}
	2 & 2   \\
	2 & 2  \\
\end{pmatrix}^2=\begin{pmatrix}
2 & 2   \\
2 & 2  \\
\end{pmatrix}.	
\end{align*}
Hence 
\begin{align*}
I-TT^*=I-T^*T=\begin{pmatrix}
	1 & 0   \\
	0 & 1  \\
\end{pmatrix}-\begin{pmatrix}
2 & 2   \\
2 & 2  \\
\end{pmatrix}=\begin{pmatrix}
2 & 1   \\
1 & 2  \\
\end{pmatrix}.		
\end{align*}
Just take 
\begin{align*}
	M_T=M_{T^*}\coloneqq \begin{pmatrix}
		2 & 1   \\
		1 & 2  \\
	\end{pmatrix}.		
\end{align*}
Then 
\begin{align*}
	M_T^2=\begin{pmatrix}
		2 & 1   \\
		1 & 2  \\
	\end{pmatrix}\begin{pmatrix}
	2 & 1   \\
	1 & 2  \\
\end{pmatrix}=\begin{pmatrix}
2 & 1   \\
1 & 2  \\
\end{pmatrix}=I-TT^*=I-T^*T
\end{align*}
and 
\begin{align*}
TM_T= \begin{pmatrix}
	2 & 2   \\
	2 & 2  \\
\end{pmatrix}\begin{pmatrix}
2 & 1   \\
1 & 2  \\
\end{pmatrix}=\begin{pmatrix}
0 &  0  \\
0 & 0  \\
\end{pmatrix}=\begin{pmatrix}
2 & 1   \\
1 & 2  \\
\end{pmatrix}\begin{pmatrix}
2 & 2   \\
2 & 2  \\
\end{pmatrix}=M_TT.						
\end{align*}
\end{example} 
\begin{example}
The operator 
\begin{align*}
	T\coloneqq 	 \begin{pmatrix}
		1 & 1   \\
		1 & 1  \\
	\end{pmatrix}: \mathbb{Z}_2^2 \ni (x,y)\mapsto T  \begin{pmatrix}
	x\\
	y\\
\end{pmatrix} \in  \mathbb{Z}_2^2
\end{align*} 
is a magic contraction. Take 
\begin{align*}
	M_T=M_{T^*}\coloneqq \begin{pmatrix}
		0 & 1   \\
		1 & 0  \\
	\end{pmatrix}.		
\end{align*}
Then 
\begin{align*}
	I-TT^*=I-T^*T=\begin{pmatrix}
		1 & 0   \\
		0 & 1  \\
	\end{pmatrix}-\begin{pmatrix}
		0 & 0   \\
		0 & 0  \\
	\end{pmatrix}=\begin{pmatrix}
		1 & 0  \\
		0 & 1  \\
	\end{pmatrix}=\begin{pmatrix}
	0 & 1   \\
	1 & 0  \\
\end{pmatrix}\begin{pmatrix}
0 & 1   \\
1 & 0  \\
\end{pmatrix}=M_T^2=M_{T^*}^2
\end{align*}
and 
\begin{align*}
	TM_T=\begin{pmatrix}
		1 & 1   \\
		1 & 1  \\
	\end{pmatrix}\begin{pmatrix}
	0 & 1   \\
	1 & 0  \\
\end{pmatrix}=\begin{pmatrix}
1 & 1   \\
1 & 1  \\
\end{pmatrix}=\begin{pmatrix}
0 & 1   \\
1 & 0  \\
\end{pmatrix}\begin{pmatrix}
1 & 1   \\
1 & 1  \\
\end{pmatrix}=M_TT.
\end{align*}
Note that we can directly verify that 
\begin{align*}
\begin{pmatrix}
	1 & 1   &0&1 \\
	1 & 1   &1&0\\
	0 & 1   &1&1 \\
	1 & 0   &1&1 \\
\end{pmatrix}	
\end{align*}
is a Halmos dilation of $T$. We can also take 
\begin{align*}
	M_T=M_{T^*}\coloneqq \begin{pmatrix}
		1 & 0   \\
		0 & 1  \\
	\end{pmatrix}.		
\end{align*}
In this case, we get the matrix
\begin{align*}
	\begin{pmatrix}
		1 & 1   &1&0 \\
		1 & 1   &0&1\\
		1&0   &1&1 \\
		0 & 1   &1&1 \\
	\end{pmatrix}	
\end{align*}
which is also a Halmos dilation of $T$.
\end{example}
\begin{example}
Let $a,b\in \mathbb{N} \cup \{0\}$  and $p\geq 2$   be such that $a^2+b^2\equiv p-1 (\operatorname{mod} p)$ and $2ab\equiv p-2 (\operatorname{mod} p)$. Then the operator 
\begin{align*}
	T\coloneqq 	 \begin{pmatrix}
		p-1 & p-1   \\
		p-1 & p-1  \\
	\end{pmatrix}: \mathbb{Z}_p^2 \ni (x,y)\mapsto T \begin{pmatrix}
		x\\
		y\\
	\end{pmatrix} \in  \mathbb{Z}_p^2
\end{align*} 
is a magic contraction. Notice
\begin{align*}
	\begin{pmatrix}
		p-1 & p-1   \\
		p-1 & p-1  \\
	\end{pmatrix}^2=\begin{pmatrix}
		2 & 2   \\
		2 & 2  \\
	\end{pmatrix}.	
\end{align*}
Hence 
\begin{align*}
	I-TT^*=I-T^*T=\begin{pmatrix}
		1 & 0   \\
		0 & 1  \\
	\end{pmatrix}-\begin{pmatrix}
		2 & 2   \\
		2 & 2  \\
	\end{pmatrix}=\begin{pmatrix}
		p-1 & p-2   \\
		p-2 & p-1  \\
	\end{pmatrix}.		
\end{align*}
Define
\begin{align*}
	M_T=M_{T^*}\coloneqq \begin{pmatrix}
		a & b   \\
		b & a  \\
	\end{pmatrix}.		
\end{align*}
Then 
\begin{align*}
	M_T^2=\begin{pmatrix}
		a & b   \\
		b & a  \\
	\end{pmatrix}\begin{pmatrix}
		a & b   \\
		b & a  \\
	\end{pmatrix}=\begin{pmatrix}
		a^2+b^2 & 2ab   \\
		2ab & a^2+b^2  \\
	\end{pmatrix}=\begin{pmatrix}
	p-1 & p-2   \\
	p-2 & p-1  \\
\end{pmatrix}=I-TT^*=I-T^*T
\end{align*}
and 
\begin{align*}
	TM_T&= \begin{pmatrix}
		p-1 & p-1   \\
		p-1 & p-1  \\
	\end{pmatrix}\begin{pmatrix}
		a &   b \\
		b & a  \\
	\end{pmatrix}=\begin{pmatrix}
		(p-1)(a+b) &  (p-1)(a+b)  \\
		(p-1)(a+b) & (p-1)(a+b)  \\
	\end{pmatrix}\\
&=\begin{pmatrix}
		a & b   \\
		b & a  \\
	\end{pmatrix}\begin{pmatrix}
		p-1 & p-1   \\
		p-1 & p-1  \\
	\end{pmatrix}=M_TT.						
\end{align*}
\end{example}
\begin{example}
	The operator $T\coloneqq 2\in \mathbb{Z}_5$ is not a contraction. This is because
	\begin{align*}
		1-TT^*=1-4=-3=2\neq 0^2, 1^2, 2^2, 3^2, 4^2.
	\end{align*}
\end{example}
\begin{example}
	Let $p$ be an odd prime and consider $\mathbb{Z}_p$.  Now $\operatorname{gcd} (2, p-1)=2$. Let  $a \in \mathbb{Z}_p$ such that   $\operatorname{gcd} (1-a^2,p)=1$  and  
	\begin{align*}
		(1-a^2)^\frac{p-1}{2}\not\equiv 1 (\operatorname{mod} p).
	\end{align*}
Quadratic reciprocity then says that $a$ is not a contraction.
\end{example}
\begin{example}
	Consider $\mathbb{C}$ with involution as identity. Let $a, b \in \mathbb{C}$ be such that $a^2+b^2=1$. Then $a$ is a magic contraction. Halmos dilation of $a$ is 
	\begin{align*}
	\begin{pmatrix}
		a & b   \\
		b & -a  \\
	\end{pmatrix}.	
	\end{align*}
Hence every complex number is a magic contraction w.r.t. identity involution! We can do this on any commutative ring whenever ring has elements $a, b$ such that $a^2+b^2=1$.
\end{example}
It is a good problem (which seems to be not easy and may require \textbf{Number Theory} tools such as \textbf{Quadratic Reciprocity}) to characterize all magic contractions in the set of all $n$ by $n$ matrices over $\mathbb{Z}_p$ where $p\in \mathbb{N}$.  Officially, we can formulate the following problem.
\begin{problem}
	\textbf{Let $\mathcal{R}$ be a $*$-ring (may be finite or infinite) and for $m, n\in \mathbb{N}$, let $\mathbb{M}_{m\times n}(\mathcal{R})$ be the set of all $m$ by $n$ matrices over $\mathcal{R}$. Let $I_n$ be the $n$ by $n$ identity  matrix over $\mathcal{R}$. Classify matrices $T \in \mathbb{M}_{m\times n}(\mathcal{R})$ which are magic contractions, i.e., for which matrices $T \in \mathbb{M}_{m\times n}(\mathcal{R})$, there are   self adjoint matrices (may not be unique)  $M_T\in \mathbb{M}_{m\times n}(\mathcal{R})$, $M_{T^*} \in \mathbb{M}_{n\times m}(\mathcal{R})$ satisfying following:
	\begin{align*}
		&M_T^2=I_n-T^*T, \quad M_{T^*}^2=I_m-TT^*,\\
		&TM_T=M_{T^*}T.
	\end{align*}
If $\mathcal{R}$ is finite, what is the number of magic contractions in $\mathbb{M}_{m\times n}(\mathcal{R})$  or whether there is atleast a good upper bound on the number of magic contractions in $\mathbb{M}_{m\times n}(\mathcal{R})$?}
\end{problem} 
  \bibliographystyle{plain}
 \bibliography{reference.bib}

\end{document}